\documentclass{amsart}
\usepackage{amssymb}
\usepackage{amscd}
\usepackage{amsmath}
\usepackage{amsfonts}
\usepackage{latexsym}
\usepackage[all]{xy}
\usepackage{verbatim}

\newtheorem{theorem}{Theorem}[section]
\newtheorem{lemma}[theorem]{Lemma}
\newtheorem{proposition}[theorem]{Proposition}
\newtheorem{corollary}[theorem]{Corollary}
\theoremstyle{definition}
\newtheorem{definition}[theorem]{Definition}
\newtheorem{example}[theorem]{Example}

\newtheorem{remark}[theorem]{Remark}

\newcommand{\id}{\text{id}}

\newcommand{\FPdim}{\text{FPdim}}

\newcommand{\Fun}{\text{Fun}}
\newcommand{\End}{\text{End}}
\renewcommand{\Vec}{\text{Vec}}

\newcommand{\Hom}{\text{Hom}}

\newcommand{\Aut}{\text{Aut}}

\newcommand{\Rep}{\text{Rep}}

\newcommand{\C}{\mathcal{C}}

\newcommand{\D}{\mathcal{D}}

\newcommand{\M}{\mathcal{M}}

\newcommand{\N}{\mathcal{N}}

\newcommand{\TY}{\mathcal{T}\mathcal{Y}}

\newcommand{\be}{\mathbf{1}}

\renewcommand{\be}{\mathbf{1}}

\newcommand{\bt}{\boxtimes}
\newcommand{\ot}{\otimes}

\begin{document}
\title[Non group-theoretical semisimple Hopf algebras]
{Non group-theoretical semisimple Hopf algebras
from group actions on fusion categories}

\author{Dmitri Nikshych}
\address{Department of Mathematics and Statistics,
University of New Hampshire,  Durham, NH 03824, USA}
\email{nikshych@math.unh.edu}

\date{December 4, 2007}

\begin{abstract}
Given an action of a finite group $G$ on a fusion category $\C$
we give a criterion for the category of $G$-equivariant objects
in $\C$ to be group-theoretical, i.e., to be categorically Morita
equivalent to a category of group-graded vector spaces.
We use this criterion to answer affirmatively the 
question about existence of non group-theoretical semisimple Hopf algebras 
asked by P.~Etingof, V.~Ostrik, and the author in \cite{ENO}.
Namely, we show that certain  $\mathbb{Z}/2\mathbb{Z}$-equivariantizations 
of fusion categories constructed by D.~Tambara and S.~Yamagami 
\cite{TY} are equivalent to representation categories
of non group-theoretical semisimple Hopf algebras.  We describe
these Hopf algebras as extensions and show that they are upper
and lower semisolvable.
\end{abstract}

\maketitle

\section{Introduction}
\label{Sect: Intro}

\subsection{Conventions} 
Throughout this article we work over an algebraically closed
field $k$ of zero characteristic. 
All cocycles appearing in this  article
have coefficients in the trivial module $k^\times$.
Module categories are assumed to  be $k$-linear and semisimple with finite-dimensional
$\Hom$-spaces and finitely many isomorphism classes of simple objects.
Functors between fusion categories and their module categories 
are assumed to be additive and $k$-linear. 
By a subcategory we always mean a full subcategory.

\subsection{Main result} 
A tensor category  
is said to be {\em group-theoretical} \cite{O2, ENO} if it is  dual to a category 
of group-graded vector spaces with respect to an indecomposable module category, 
see Section~\ref{Sect: Prelim} for definitions.  Group-theoretical 
categories can be  explicitly described in terms of finite groups and their cohomology, 
see \cite{O2} and Section~\ref{sect: group-theor} below.  Such categories
were extensively studied, see e.g., \cite{DGNO, EGO, ENO, GN, Na, NaNi, Nt},  
and many classification results were obtained. In particular,
representation categories of semisimple quasi-Hopf algebras of dimension 
$p^n,\, n=1,2,\dots$ and $pq$,  where $p$ and $q$ are prime numbers, 
are group-theoretical by \cite{DGNO} and \cite{EGO}, respectively.
On the other hand, there exist semisimple quasi-Hopf algebras of dimension
$2p^2$, where $p$ is an odd prime,  
with non group-theoretical representation categories \cite[Remark 8.48]{ENO}.

The  main result of this article is existence  of semisimple
Hopf algebras with non group-theoretical representation categories. 
This answers a question \cite[Question 8.45]{ENO} of P.~Etingof, V.~Ostrik, 
and the author. 

To establish this we construct a series of fusion categories of Frobenius-Perron dimension $4p^2$,
where $p$ is an odd  prime, that are non group-theoretical and admit tensor
functors to the category of vector spaces. By Tannakian formalism
\cite{U} these categories  are equivalent to 
representation categories of some semisimple Hopf algebras.  The construction
of these categories involves an {\em equivariantization} procedure \cite{AG, G}: 
given an action of a finite  group $G$ on a tensor category $\C$ one forms a tensor 
category $\C^G$ of $G$-equivariant objects in $\C$. See Section~\ref{equivar}
for definitions. Specifically, we  take $\C$ to be a group-theoretical 
Tambara-Yamagami fusion category \cite{TY}  of Frobenius-Perron dimension $2p^2$  
and explicitly define  an action  of $G =\mathbb{Z}/2\mathbb{Z}$ on it.
The corresponding category  of equivariant objects is equivalent 
to the representation category of a non group-theoretical semisimple 
Hopf algebra of dimension $4p^2$.

\subsection{Organization of the paper}
Section~\ref{Prelim} contains
basic definitions related to fusion categories and their module categories,
a description of group-theoretical categories, the equivariantization 
construction, and Tambara-Yamagami fusion categories.  
In Section~\ref{Sect: cross}, given an action of a group $G$ on a fusion category $\C$,  
we define a {\em crossed product} category $\C\rtimes G$ 
and show that it is dual to the category $\C^G$ of $G$-equivariant 
objects in $\C$. We prove in Theorem~\ref{criterion for C times G to be gr th} 
that $\C\rtimes G$ and $\C^G$ are group-theoretical if and only if there exists
a $G$-invariant indecomposable $\C$-module category with a pointed dual. 
This criterion is used
in Section~\ref{sect: construction}, where we explicitly construct an action of 
$\mathbb{Z}/2\mathbb{Z}$ on a group-theoretical Tambara-Yamagami category $\C_p$
of Frobenius-Perron dimension $2p^2$, where $p$ is an odd prime, 
and show that there are no $\C_p$-module categories with pointed dual
invariant with respect to this action. 
The equivariantization
categories  $\C_p^{\mathbb{Z}/2\mathbb{Z}}$
are non group-theoretical but admit tensor functors to the category
of vector spaces. Hence,  these categories are
equivalent to representation categories
of semisimple Hopf algebras. This shows that there is a series of semisimple
Hopf algebras $H_p$ of dimension $4p^2$,  where $p$ is an odd prime, 
with  non group-theoretical representation categories.
The algebraic structure of these Hopf algebras
is analyzed in Section~\ref{Hopf}. We describe $H_p$  as 
an extension of already known Hopf algebras and show that it is upper 
and lower semisolvable in the sense of \cite{MW}.

\subsection{Acknowledgments} 
The author's research was supported by the NSA grant H98230-07-1-0081. 
He is grateful to Pavel Etingof, Shlomo Gelaki, 
Victor Ostrik, and Leonid Vainerman for useful discussions.

\section{Preliminaries}
\label{Prelim}

\subsection{Fusion categories and their module categories}
\label{Sect: Prelim}

A {\em fusion category} over $k$ is a
$k$-linear semisimple rigid tensor category \cite{BK, K} with finitely many
isomorphism classes of simple objects, finite-dimensional Hom-spaces,
and simple unit object~$\mathbf{1}$.

Let $\C$ be a fusion category. For any object $X$ in $\C$ its
{\em Frobenius-Perron dimension} $\FPdim(X)$ is defined as the largest non-negative
real eigenvalue of the matrix of multiplication by $X$ in the Grothendieck
semi-ring of $\C$, see \cite{FK} and \cite[Section 8]{ENO}.  
The Frobenius-Perron dimension of $\C$  is, by definition, 
the sum of squares of  Frobenius-Perron dimensions of simple objects 
of $\C$ and is denoted $\FPdim(\C)$. In the special case when $\C=\Rep(H)$ 
is the representation category of a semisimple quasi-Hopf algebra $H$
the Frobenius-Perron dimensions coincide with vector space dimensions,
i.e., $\FPdim(V) =\dim_k(V)$ for any finite-dimensional $H$-module $V$
and $\FPdim(\Rep(H))=\dim_k(H)$.

A fusion category $\C$ is {\em graded} 
by a finite group $G$ if there is a decomposition
\[
\C =\oplus_{g\in G}\, \C_g
\]
of $\C$ into a direct sum of full Abelian subcategories such that
$\otimes$ maps $\C_g\times \C_h$ to $\C_{gh}$ for all $g,h\in G$. 
Note that $\C_e$, where $e$ is the identity element of $G$, 
is the fusion subcategory of $\C$. In this paper we consider only faithful
gradings, i.e., such that $\C_g\neq 0$ for all $g\in G$. 
In this case one has $\FPdim(\C) =  |G| \FPdim(\C_e)$, see \cite{ENO}.

A right {\em $\C$-module category} is a category $\M$ together with 
an exact bifunctor $\ot: \M \times \C \to \M$ and a natural family 
of isomorphisms $M \ot (X \ot Y) \cong (M \ot X) \ot Y$ and
$M \ot \be \cong M,\, X,Y\in \C, \, M\in \M$ satisfying certain coherence 
conditions. See \cite{O1} for details and for definitions of 
$\C$-module functors and their natural transformations. 
The category of $\C$-module functors between right $\C$-module categories
$\M_1$ and $\M_2$ will be denoted $\Fun_\C(\M_1,\, \M_2)$.

For two $\C$-module categories $\M_1$ and $\M_2$ 
their {\em direct sum} $\M_1 \oplus \M_2$ has
an obvious $\C$-module category structure. A module category is 
{\em indecomposable} if it is not equivalent to a
direct sum of two non-trivial module categories.
It was shown in \cite{O1} that $\C$-module categories are completely 
reducible, i.e., given a $\C$-module subcategory $\N$ of a
$\C$-module category $\M$ there is a unique $\C$-module subcategory
$\N'$ of $\M$ such that $\M =\N \oplus \N'$. Consequently,
any $\C$-module category $\M$ has a unique, up to a permutation
of summands, decomposition $\M =\oplus_{x\in S}\,\M_x$
into a direct sum of indecomposable $\C$-module categories.

A theorem due to V.~Ostrik \cite{O1} states that any right $\C$-module
category is equivalent to the category of left modules over some algebra
in $\C$. Namely, given a non-zero object $V$ of an indecomposable
$\C$-module category $\M$, the internal Hom $\underline{\Hom}(V, V)$
defined by a natural isomorphism
\begin{equation}
\label{int Hom}
\Hom_\M(X\ot V,\, V) \cong \Hom_\C(X,\, \underline{\Hom}(V, V)),\qquad X\in \C,
\end{equation}
is an algebra in $\C$ and the category of its left modules in $\C$
is equivalent to $\M$. 

An important fact in the theory of module categories is that the category
$\C^*_\M: = \Fun_\C(\M,\, \M)$ of $\C$-module endofunctors of an indecomposable 
$\C$-module category  $\M$ (called the {\em dual} of $\C$ with respect to $\M$)
is also a fusion category and 
\[
\FPdim(\C^*_\M) = \FPdim(\C).
\] 
Furthermore, $\M$ is an indecomposable left $\C^*_\M$-module category.
If $\M$ is the category of left modules over an algebra $A$ in $\C$ then
$\C^*_\M$ is equivalent to the fusion category of $A$-bimodules in $\C$
with the tensor product $\ot_A$, see \cite{O1}.

Following M.~M\"uger \cite{Mu} we will say that two fusion categories
$\C$ and $\D$ are {\em Morita equivalent} if there is a right $\C$-module category 
$\M$ such that $\D \cong \C^*_\M$. It was shown in \cite{Mu} that the
above relation is indeed an equivalence.  For a fixed $\M$ the assignment
\begin{equation}
\label{ module cat coorrespond}
\N \mapsto \Fun_\C(\M,\, \N)
\end{equation}
establishes an equivalence between $2$-categories of right $\C$-module categories
and right $\C^*_\M$-module categories.

Define a {\em rank} of a semisimple category $\mathcal{A}$ 
to be the number of equivalence classes of simple objects in $\mathcal{A}$. 
A fusion category of rank $1$ is equivalent to the category $\Vec$ of 
$k$-vector spaces.  A module category of rank $1$ over a fusion category $\C$
is the same thing as a {\em fiber functor}, i.e., a tensor functor $F: \C\to \Vec$.
From such a functor $F$  one obtains a semisimple Hopf algebra $H:=\End(F)$
such that $\C$ is equivalent to the category  $\Rep(H)$ 
of finite-dimensional representations of $H$ \cite{U}.

A fusion category is called {\em pointed} if every its simple object
is invertible.  Every pointed fusion category is equivalent to a category
$\Vec_G^\omega$, where $G$ is a finite group and $\omega \in Z^3(G,\, k^\times)$
is a $3$-cocycle. By definition, the latter is the category of $G$-graded vector 
spaces with the associativity constraint given by $\omega$. 
The simple objects of  $\Vec_G^\omega$ are $1$-dimensional 
$G$-graded vector spaces which will be denoted by $g,\, g\in G$. Note that
the rank  and Frobenius-Perron dimension of $\Vec_G^\omega$ are equal to $|G|$.

\begin{definition}
\label{ptd module}
Let $\C$ be a fusion category and let $\M$ be an indecomposable $\C$-module category.
We will say that $\M$ is {\em pointed} if $\C^*_\M$ is pointed. 
\end{definition}

\begin{lemma}
\label{homogen}
Let $\C$ be a fusion category and let $\D \subseteq \C$ be a fusion subcategory.
Let $\M$ be a pointed $\C$-module category and let $\M =\oplus_{x\in S}\, \M_x$
be its decomposition into a direct sum of indecomposable $\D$-module categories.
Then $\M_x\cong \M_y$ as $\D$-module categories for all $x,y\in S$.
\end{lemma}
\begin{proof}
Any $\C$-module autoequivalence of $\M$ induces a permutation of $S$.
Since $\M$ is indecomposable as a module category over the pointed
category $\C^*_\M$, the  simple objects of $\C^*_\M$ act transitively on $S$.
Hence, for any pair $x,y\in S$  there is a $\C$-module autoequivalence 
$F:\M\to \M$ that maps $\M_x$ to $\M_y$.
\end{proof}

\subsection{Group-theoretical fusion categories}
\label{sect: group-theor}

A fusion category that has a pointed module category is called
{\em group-theoretical} \cite{O2, ENO}.  Any such category  is equivalent
to the category of bimodules over an algebra in $\Vec_G^\omega$ for some
$G$ and $\omega \in Z^3(G,\, k^\times)$. 
Below we recall a description of group-theoretical categories from \cite{O2}.

Let $G$ be a finite group and $\omega \in Z^3(G, \, k^\times)$.
Equivalence classes of indecomposable right $\Vec_G^{\omega}$-module categories 
correspond to  pairs
$(H, \, \mu)$, where $H$ is a subgroup of $G$ such that
$\omega|_{H \times H \times H}$ is cohomologically trivial
and $\mu \in C^2(H, \, k^\times)$ is a $2$-cochain satisfying 
$\delta^2\mu = \omega|_{H \times H \times H}$. The corresponding
$\Vec_G^{\omega}$-module category is constructed as follows.
Given a pair $(H, \, \mu)$ as above define an algebra 
\[
R(H, \mu) =\bigoplus_{a\in H} a
\]
in $\Vec_G^\omega$ with the multiplication
\begin{equation}
\label{algebra R}
\bigoplus_{a,b\in H}\, \mu(a,\,b) \id_{ab} :  
R(H,\, \mu) \ot R(H,\, \mu) \to R(H,\, \mu).
\end{equation}
Let $\M(H,\, \mu)$ denote the category of left $R(H,\, \mu)$-modules in $\Vec_G^\omega$.
Any $\Vec_G^\omega$-module category is equivalent to some $\M(H,\, \mu)$.
The rank of $\M(H,\, \mu)$ is equal to the index of $H$ in $G$.
Two $\Vec_G^\omega$-module categories $\M(H,\, \mu)$ and $\M(H',\, \mu')$ 
are equivalent if and only if there is $g\in G$ such that
$H'=gHg^{-1}$  and $\mu$ and the $g$-conjugate of $\mu'$ differ by
a coboundary.

Let us analyze  $(\Vec_G^\omega)^*_{\M(H,\, \mu)}$-module categories 
using correspondence \eqref{ module cat coorrespond}. 
Let $(H_1,\, \mu_1)$ and $(H_2,\, \mu_2)$ 
be two pairs as above. The rank of the semisimple category
$\Fun_{\Vec_G^\omega}(\M(H_1,\, \mu_1),\, \M(H_2,\, \mu_2))$ 
was computed in \cite[Proposition 3.1]{O2}.  
Namely, for any $g\in G$ the group $H^g:= H_1\cap gH_2g^{-1}$ 
has a well-defined $2$-cocycle
\begin{gather*}
\mu^g(h,h'):= \mu_1(h,h') \mu_2(g^{-1}h'^{-1}g,\, g^{-1}h^{-1}g)
\omega(hh'g,\, g^{-1}h'^{-1}g,\, g^{-1}h^{-1}g)^{-1} \times \\
 \omega(h,\,h',\,g) \omega(h,\, h'g,\, g^{-1}h'^{-1}g),\qquad h,h'\in H^g.
\end{gather*}
The simple objects of $\Fun_{\Vec_G^\omega}(\M(H_1,\, \mu_1),\, \M(H_2,\, \mu_2))$
correspond to pairs $(Z,\, \pi)$, where $Z$ is a two-sided $(H_1,\, H_2)$-coset 
in $G$ and $\pi$ is an irreducible projective representation 
of $H^{g}$ with the Schur multiplier $\mu^{g},\, g\in Z$. 
In particular,
\begin{equation}
\label{rank 12}
\mbox{rank}(\Fun_{\Vec_G^\omega}(\M(H_1,\, \mu_1),\, \M(H_2,\, \mu_2)))
= \sum_{i\in H_1\backslash G  / H_2}\, m(g_i),
\end{equation}
where $\{g_i\}_{i\in H_1\backslash G  / H_2}$ is a set of representatives of 
two-sided $(H_1,\, H_2)$-cosets in $G$ and
$m(g_i)$ is the number of non-equivalent irreducible projective representations
of the group $H^{g_i}$ with the Schur multiplier $\mu^{g_i}$.  Note that
$m(g_i)$ is independent from the choice of a coset representative $g_i$.

\begin{remark}
\label{rank consequences}
Let us note several consequences of the rank formula \eqref{rank 12}. 
\begin{enumerate}
\item[(i)] The category $(\Vec_G^\omega)_{\M(H,\, \mu)}^*$ is equivalent
to the representation category of a Hopf algebra  if and only if   there
is pair $(K,\, \nu)$, where $K$ is a subgroup of $G$  and $\mu\in H^2(K,\, k^\times)$  
such that $\omega|_{K\times K\times K}$ 
is trivial, $HK = G$, and $\mu\nu^{-1}|_{H \cap K}$ is
non-degenerate. The latter condition means that the group algebra of $H \cap K$
twisted by a cocycle representing $\mu\nu^{-1}|_{H \cap K}$ is simple.

Moreover, the conjugacy classes of pairs $(K,\, \nu)$ with the above properties 
parameterize fiber functors of the category $(\Vec_G^\omega)_{\M(H,\, \mu)}^*$.
\item[(ii)]  Simple objects of $(\Vec_G^\omega)_{\M(H,\, \mu)}^*$
correspond to pairs  $(Z,\, \pi)$, where $Z$ is a two-sided 
coset of $H$ in $G$ and $\pi$ is an irreducible projective representation 
of $H\cap H^g$ with the Schur multiplier $\mu^{g},\, g\in Z$.
The  Frobenius-Perron dimension of the simple object
corresponding to $(Z,\, \pi)$ is $\deg(\pi) (|Z|/|H|)$. 
\item[(iii)] The pointed $(\Vec_G^\omega)_{\M(H,\, \mu)}^*$-module categories
are 
\[
\Fun_{\Vec_G^\omega}(\M(H,\, \mu),\, \M(N,\, \nu)), 
\]
where $N$ is a normal Abelian subgroup of $G$ such that $\omega|_{N\times N\times N}$ 
is trivial and $\nu\in H^2(N,\, k^\times)$ is a $G$-invariant cohomology class, 
see \cite{Na}.
\end{enumerate}
\end{remark}

\subsection{Group actions on fusion categories and equivariantization}
\label{equivar}

Let us recall the following well known construction (see \cite{AG, G}):
given a fusion category $\C$ equipped with an action of a finite group $G$
one defines a new fusion category, namely {\em the category of $G$-equivariant
objects\,} of $\C$, also known as the  {\em equivariantization\,} of $\C$.
  
Let $\C$ be a fusion category. Consider the category $\underline{\Aut}_\ot(\C)$, whose objects are
tensor auto-equivalences of $\C$ and whose morphisms are isomorphisms of tensor functors.
The category $\underline{\Aut}_\ot(\C)$ has an obvious structure of monoidal category,
in which the tensor product is the composition of tensor functors.  

For a finite group $G$ let $\mbox{Cat}(G)$ denote the monoidal category
whose objects objects are elements of $G$, the only morphisms are the identities, 
and the tensor product is given by multiplication in $G$. 

\begin{definition} \label{action}
An {\em action} of a group $G$ on a fusion category $\C$ is a monoidal functor
\[
\mbox{Cat}(G)\to \underline{Aut}_\ot(\C) : g \mapsto T_g.
\]
In this situation we also say that $G$ {\em acts} on $\C$.
\end{definition} 

Let $G$ be a finite group acting on a fusion category $\C$. For any $g\in G$ let $T_g\in 
\underline{\Aut}_\ot(\C)$ be the corresponding functor and for any $g,h\in G$ let $\gamma_{g,h}$ be
the isomorphism $T_g\circ T_h\simeq T_{gh}$ that defines the
monoidal structure on the functor $\mbox{Cat}(G)\to \underline{\Aut}_\ot(\C)$. 

\begin{definition}
\label{Gequiv object}
A {\em $G$-equivariant object\,} in $\C$ is
an object $X$ of $\C$ together with isomorphisms 
$u_g: T_g(X)\simeq X, g\in G,$ such that the diagram
\begin{equation*}
\label{equivariantX}
\xymatrix{T_g(T_h(X))\ar[rr]^{T_g(u_h)} \ar[d]_{\gamma_{g,h}(X)
}&&T_g(X)\ar[d]^{u_g}\\ T_{gh}(X)\ar[rr]^{u_{gh}}&&X}
\end{equation*}
commutes for all $g,h\in G$.
One defines morphisms of equivariant objects to be morphisms in $\C$ commuting with $u_g,\; g\in G$.
\end{definition}

The {\em category of $G$-equivariant objects of $\C$}, or {\em equivariantization},
will be denoted by $\C^G$. 
It has an obvious  structure of a fusion  category and a forgetful tensor functor 
$\mbox{Forg}: \C^G\to \C$.

\begin{remark}
\label{fiber}
If $\C$ has a fiber functor $F: \C \to \Vec$ then $F\circ \mbox{Forg}$
is a fiber functor of $\C^G$. Thus, if $\C$ is a representation category
of a Hopf algebra then so is $\C^G$.
\end{remark}

\begin{example}
\label{crossed product}
Let $G,\,K$ be finite groups such that $G$ acts on $K$ by automorphisms.
Then $G$ acts on the category $\Rep(K)$ of representations of $K$
and $\Rep(K)^G \cong \Rep(K \rtimes G)$.  
\end{example}


\subsection{Tambara-Yamagami categories}
\label{sect: definition of TY}

In \cite{TY} D.~Tambara and S.~Yamagami completely classified all
$\mathbb{Z}/2\mathbb{Z}$-graded fusion categories in which all but one
of simple objects are invertible. They showed that any such category
$\TY(A,\chi, \tau)$ is determined, up to an equivalence, by a
finite Abelian group $A$, an isomorphism class of a  
non-degenerate symmetric bilinear form
$\chi: A\times A \to k^\times$, and a number $\tau\in k$ such that
$\tau^2= |A|^{-1}$. The category $\TY(A,\chi, \tau)$ is described as
follows. It is a skeletal category (i.e., such that 
isomorphic objects are equal) with simple objects $a, a\in
A$, and $m$  and tensor product
\[
a\ot b =a+b,\quad a\ot m = m,\quad m \ot a=m,\quad m \ot m
=\oplus_{a\in A}\, a,
\]
for all $a, b\in A$ and the unit object $0\in A$.
The associativity constraints 
\[
\alpha_{x,y,z} : (x\ot y)\ot z \cong  x \ot (y \ot z),
\]
where $x,y,z$ are objects of $\TY(A,\chi, \tau)$, are given by
\begin{eqnarray*}
\alpha_{a, b, c} &=& \id_{a+b+c},  \qquad \\
\alpha_{a, b, m} &=& \id_{m}, \\
\alpha_{a, m, b} &=& \chi(a,b)\,\id_{m}, \qquad \\
\alpha_{m, a, b} &=& \id_{m}, \\
\alpha_{a, m, m} &=& \oplus_{b\in A}\, \id_{b}, \qquad \\
\alpha_{m, a, m} &=& \oplus_{b\in A}\, \chi(a,b)\,\id_{b}, \\
\alpha_{m, m, a} &=& \oplus_{b\in A}\, \id_{b}, \qquad \\
\alpha_{m, m, m} &=& \oplus_{a,b\in A}\, \tau\chi(a,b)^{-1}\,\id_{m}.
\end{eqnarray*}
The unit constraints are the identity maps.
The category $\TY(A,\chi,\tau)$ is rigid with $a^*=-a$ and $m^*=m$.

The Frobenius-Perron dimensions of simple objects of
$\TY(A,\chi,\tau)$ are $\FPdim(a)=1,\, a\in A$, and
$\FPdim(m)=\sqrt{|A|}$. We have $\FPdim(\TY(A,\chi,\tau))=2|A|$.

\begin{definition}
\label{hyperbolic}
Let $A$ be an Abelian group and let   $\chi: A \times A \to k^\times$ be 
a non-degenerate symmetric bilinear form on it.  The form $\chi$ is called
{\em hyperbolic} if there are subgroups $L,\, L'$ of $A$ such that
$A=L \times L'$ and 
\[
\chi|_{L \times L} = \chi|_{L' \times L'}  = 1.
\]
Any subgroup $L$ of $A$ for which there is $L'$ with the above properties is
called {\em Lagrangian} (with respect to $\chi$).
\end{definition}

\begin{remark}
\label{Tambara fiber}
Suppose that $|A|$ is odd. It was shown by D.~Tambara in \cite{T} 
that $\TY(A,\chi,\tau)$ admits a fiber functor (i.e., $\TY(A,\chi,\tau)$
is equivalent to the representation category of a semisimple Hopf algebra)
if and only if $\tau^{-1}$ is a positive integer and $\chi$ is hyperbolic.

In the special case when $A=\mathbb{Z}/k\mathbb{Z}\times \mathbb{Z}/k\mathbb{Z},
\, k\geq 2$, the corresponding semisimple Hopf algebras were explicitly described
by G.~Kac and V.~Paljutkin \cite{KP}.
\end{remark}

\begin{proposition}
\label{orthogonal group acts}
Let $\Aut(A, \chi)$ be the group of automorphisms of $A$ preserving the form
$\chi$.
There is an action $g\mapsto T_g$ of $\Aut(A, \chi)$ on $\TY(A,\chi,\tau)$,
where
\[
T_g(A) = g(a),\quad T_g(m)=m,\qquad a\in A,\, g\in \Aut(A, \chi),
\]
with the tensor structure of $T_g$ given by identity morphisms.
\end{proposition}
\begin{proof}
This follows directly from the definition of the associativity constraints
in the Tambara-Yamagami category and an observation that 
$\alpha_{T_g(x),T_g(y), T_g(z)} = T_g (\alpha_{x,y,z})$
for all simple objects $x,y,z$ in $\TY(A,\chi,\tau)$ and all $g\in \Aut(A, \chi)$.
\end{proof}

\section{Crossed product fusion categories and their module categories}
\label{Sect: cross}


In this Section we construct a fusion category 
$\C \rtimes G$  dual to the equivariantization category $\C^G$ 
(see Proposition~\ref{CG duality} below)
with respect to the module category $\C$ and 
derive a criterion for $\C^G$ and $\C \rtimes G$
to be group-theoretical.

For a pair of Abelian categories $\mathcal{A}_1$,  $\mathcal{A}_2$,
let $\mathcal{A}_1 \bt \mathcal{A}_2$ denote their Deligne's
tensor product \cite{D}. 

\subsection{Definition of a crossed product fusion category}

Let $\C$ be a fusion category.
Fix a finite group $G$ and an  action 
$\mbox{Cat}(G)\to \underline{Aut}_\ot(\C): g \mapsto T_g$.

\begin{definition}
A {\em crossed product} category $\C \rtimes G$ is defined as follows. 
We set $\C \rtimes  G = \C \bt \Vec_G$ as an Abelian category, and define
a tensor product by
\begin{equation}
(X \bt g) \ot (Y \bt h): = (X \ot T_g(Y)) \bt gh,\qquad X,Y \in \C,\quad g,h\in G.
\end{equation}
The unit object is $\be \bt e$ and the associativity and unit constraints
come from those of $\C$.
\end{definition}

Note that $\C \rtimes G$ is a $G$-graded fusion category,
\[
\C \rtimes G =\bigoplus_{g\in G}\, (\C \rtimes G)_g,\qquad \mbox{ where }     
(\C \rtimes G)_g= \C \ot (\be \bt g).
\]
In particular, $\C \rtimes G$ contains $\C =\C \ot (\be \bt e)$ as a fusion subcategory.

We have $\FPdim(\C \rtimes G) = |G| \FPdim(\C)$.

There is a left  $(\C \rtimes G)$-module category structure on $\C$ given by
\[
V\ot (X \bt g) := T_{g^{-1}}(V \ot X),\qquad V,X\in \C,\, g\in G,
\]
with the associativity constraint 
\begin{eqnarray*}
\alpha_{V, X\bt g , Y\bt h} : V \ot ((X \bt g) \ot (Y \bt h))
&=& T_{h^{-1}g^{-1}}(V \ot (X \ot T_g(Y)))\\
&\cong& T_{h^{-1}}( T_{g^{-1}}(V \ot X) \ot Y) \\
&=&  T_{g^{-1}}(V \ot X) \ot (Y\bt h)\\
&=&  (V \ot (X \bt g)) \ot (Y\bt h). 
\end{eqnarray*}

\begin{proposition}
\label{CG duality}
We have $(\C \rtimes G)^*_\C \cong \C^G$, i.e., the categories 
$(\C \rtimes G)$ and $\C^G$ are Morita equivalent.
\end{proposition}
\begin{proof}
Let $F: \C \to \C$ be a $(\C \rtimes G)$-module functor.
In particular, $F$ is a $\C$-module functor, hence $F(V) =X\ot V$
for some $X$ in $\C$. It is straightforward to check that 
a $(\C \rtimes G)$-module functor structure on
the latter functor is the same thing as a $G$-equivariant object
structure on $X$.
\end{proof}

\subsection{A criterion for a crossed product fusion category to be group theoretical}

Let $\C$ be a fusion category, let $t \in \underline{\Aut}_\ot(\C)$ be a tensor
autoequivalence of $\C$, and let $\M$ be a $\C$-module category.

Let $\M^t$  denote the module category
obtained from $\M$ by twisting the multiplication by means of $t$, i.e.,
by defining a new action of $\C$: 
\[
M \ot^t X :=  M \ot t(X),
\]
for all objects $M$ in $\M$ and $X$ in $\C$.
If $A$ is an algebra in $\C$ such that $\M$ is equivalent to the category of 
$A$-modules in $\C$ then $\M^t$ is equivalent to the category of
$t(A)$-modules in $\C$.

\begin{lemma}
\label{dual preserved}
Let $\M$ be a $\C$-module category and let $t \in \underline{\Aut}_\ot(\C)$ be a tensor
autoequivalence of $\C$. Then $\C^*_{\M} \cong \C^*_{\M^t}$.
In particular, if $\M$ is pointed, then so is~$\M^t$.
\end{lemma}
\begin{proof}
Let $F: \M \to \M$ be a $\C$-module functor with the $\C$-module functor structure
given by
\[
\gamma_{M,X} : F(M \ot X)  \cong F(M)\ot X,\qquad M\in \M,\, X\in \C.
\]
Let $F^t: \M^t \to \M^t$ be a $\C$-module functor defined by 
$F^t (M)=F(M)$ with a $\C$-module functor structure
\[
\gamma^t_{M,X} : F^t(M \ot^t X)
=  F(M \ot t(X)) \xrightarrow{\gamma_{M, t(X)}} F(M)\ot t(X) =   F^t(M) \ot^t X.
\]
It is straightforward to check that $F\mapsto F^t$ is a tensor equivalence
between $\C^*_\M$ and $\C^*_{\M^t}$.
\end{proof}

Given an action $g\mapsto T_g$ of a group $G$ on $\C$ we will denote
by $\M^g$ the category $\M^{T_g}$. We have $\M^{gh} \cong(\M^g)^h,\, g,h\in G$, i.e.,
$G$ acts on the set  of indecomposable $\C$-module categories.

\begin{definition}
\label{G invariant module}
We will say that a $\C$-module category $\M$ is $G$-invariant
if $\M \cong \M^g$ for every $g\in G$.
\end{definition}

\begin{theorem}
\label{criterion for C times G to be gr th}
The category $\C \rtimes G$ is group-theoretical if and only if 
there exists a $G$-invariant pointed  $\C$-module category.
\end{theorem}
\begin{proof}
Let $\M$ be a $(\C \rtimes G)$-module category and let $\M =\oplus_{x\in S}\, \M_x$ 
be a decomposition of $\M$ into a direct sum of indecomposable $\C$-module categories.
Note that $\M_{gx}:=\M_x \ot (\be\bt g),\, x\in S,\,g\in G,$ is a $\C$-module subcategory
of $\M$.  
This makes $S$ a transitive $G$-set. The functor $M \mapsto M \ot (\be\bt g)$
is a $\C$-module equivalence between  $(\M_{x})^g$ and $\M_{gx}$.

Suppose that $\M$ is pointed, then $\M_x\cong \M_y$ for all $x,y\in S$
by Lemma~\ref{homogen}. Thus, for any $x\in S$ the category 
$\M_x$ is $G$-invariant.

By \cite[Proposition 5.3]{ENO} there is a surjective tensor functor
\[
(\C \rtimes G)^*_\M \to \C^*_\M = \bigoplus_{x,y\in S}\, \Fun_\C(\M_x,\, \M_y).
\]
This means, in particular, that every simple object in $\C^*_{\M_x}\,(x\in S)$
is a subobject of the image of an invertible object in $(\C \rtimes G)^*_\M$
and, hence, is invertible. Thus, $\M_x$ is pointed.

To prove the converse implication, suppose that  $\N$ is a $G$-invariant 
pointed  $\C$-module category. Choose  right $\C$-module equivalences
$K_g : \N^g \cong \N ,\, g\in G$. We have natural isomorphisms
\[
K_g(N\ot X) \cong K_g(N)\ot T_g^{-1}(X),\qquad N\in \N,\, X\in \C.
\]
We equip $\M := \N \boxtimes \Vec_G$ with a 
$(\C \rtimes G)$-module category structure by setting
\[
(N \bt f) \ot (X \bt g) := K_{fg} K_f^{-1} (N \ot X) \bt fg,\qquad
N\in \N,\, X\in \C,\, f,g\in G.
\]
The associativity constraint of $\M$
is defined via the following isomorphisms:
\begin{eqnarray*}
(N \bt f) \ot ( (X \bt g) \ot (Y \bt h)  )
&= & K_{fgh} K_f^{-1} (N \ot (X \ot T_g(Y))) \bt fgh \\
&\cong & K_{fgh} K_f^{-1} ((N \ot X )\ot T_g(Y)) \bt fgh \\
&\cong & K_{fgh} K_{fg}^{-1} ( K_{fg} K_f^{-1}(N \ot X ) \ot Y)  \bt fgh \\
&= & ((N \bt f) \ot (X \bt g)) \ot (Y \bt h).
\end{eqnarray*}
Let us consider the category $(\C \rtimes G)^*_\M$. For any right
$\C$-module equivalence $F~:~\N \cong \N^h,\, h\in G,$ define a functor
$F_h: \M \to \M$ by
\[
F_h(N \bt f) = K_{hf} F K_f^{-1}(N)  \bt hf,\qquad N\in \N,\, f,h\in G.
\]
Then $F_h$ has a structure of a right  $(\C \rtimes G)$-module
endofunctor of $\M$:
\begin{eqnarray*}
F_h((N \bt f) \ot (X \bt g))
&=& K_{hfg} F K_f^{-1}(N \ot X) \bt hfg\\
&\cong& K_{hfg} F( K_f^{-1}(N) \ot T_f(X)) \bt hfg\\
&\cong& K_{hfg} K_{hf}^{-1} K_{hf} (F(K_f^{-1}(N)) \ot T_{hf}(X)) \bt hfg\\
&\cong& K_{hfg} K_{hf}^{-1} (K_{hf} F K_f^{-1}(N) \ot X) \bt hfg\\
&=& ((K_{hf} F K_f^{-1}(N) \ot X) \bt hf) \ot (X \bt g) \\
&=&  F_h(N \bt f)  \ot (X \bt g).
\end{eqnarray*}
Each functor $F_h$ is an equivalence. Since there are $|G|\FPdim(\C)$
such functors,  we conclude that  $(\C \rtimes G)^*_\M$
is spanned by invertible objects, i.e., it is pointed.
\end{proof}

\begin{corollary}
\label{criterion for CG to be gr th}
The category $\C^G$ is group-theoretical if and only if there exists 
a $G$-invariant pointed  $\C$-module category.
\end{corollary}

\section{Construction of a series of non group-theoretical  fusion categories
from Tambara-Yamagami categories}
\label{sect: construction}

Let $p$ be an odd prime.

\subsection{A group-theoretical category $\mathbf{\C_p}$}
\label{Cp}

Let $G := D_{2p} \times \mathbb{Z}/p\mathbb{Z}$, where $D_{2p}$ is the dihedral
group of order $2p$. Let $K$ be a non-normal subgroup of $G$ of order~$p$. 
Let 
\[
\C_p:= (\Vec_G)^*_{\M(K, 1)}.
\]
Note that the centralizer of $K$ is the unique Sylow $p$-subgroup of $G$.
It follows from Section~\ref{sect: group-theor} (see \cite{O2}) that
the category $\C_p$ has $p^2$ invertible objects
and a unique simple object of Frobenius-Perron dimension $p$.
Thus, $\C_p$ is a Tambara-Yamagami fusion category. 

Note that $\C_p$ admits a fiber functor,
since $G= D_{2p}K$ and $D_{2p}\cap K =\{ e\}$, see Remark~\ref{rank consequences}(i). 
It follows from Remark~\ref{Tambara fiber} (see \cite{T}) that
\[
\C_p \cong   \TY(\mathbb{Z}/p\mathbb{Z}
\times \mathbb{Z}/p\mathbb{Z},\, \chi,\, \tfrac{1}{p}),
\]
where $\chi$ is a non-degenerate hyperbolic bilinear form on $\mathbb{Z}/p\mathbb{Z}
\times \mathbb{Z}/p\mathbb{Z}$.
Note that such $\chi$ is unique up to an automorphism of  $\mathbb{Z}/p\mathbb{Z}
\times \mathbb{Z}/p\mathbb{Z}$ and so  $\C_p$ does not depend on the choice
of $K$.

\subsection{Pointed $\mathbf{\C_p}$-module categories}

\begin{proposition}
\label{Cp module categories}
$\C_p$ has exactly $4$ non-equivalent pointed module categories. Two of these
categories have  rank $2p$ and two others have rank $2$.
\end{proposition}
\begin{proof}
Recall from  Remark~\ref{rank consequences}(ii) (see \cite{Na}) that 
pointed $\Vec_G$-module categories  correspond to pairs $(H,\, \nu)$,
where $H$ is a normal Abelian subgroup of $G$ and $\nu\in H^2(H, k^\times)$
is a $G$-invariant cohomology class. The normal Abelian subgroups
of $G$ are  the following: $\{e\}$, two normal subgroups of order $p$ (denoted $H_1,\, H_2$),
and the Sylow $p$-subgroup $P$. The subgroups $\{e\},\, H_1,\, H_2$ have trivial
second cohomology, and the only $G$-invariant cohomology class in $P$
is the trivial one. Hence,
\[
\M(\{e\},\, 1),\, \M(H_1,\, 1),\, \M(H_2,\, 1),\, \M(P,\,1)
\]
are all the pointed $\Vec_G$-module categories. 

The corresponding $\C_p$-module categories and their ranks are 
found using Remark~\ref{rank consequences}:
\begin{eqnarray*}
\label{1st ptd}
\mbox{rank}(\Fun_{\Vec_G}(\M(\{e\},\, 1),\, \M(K,\, 1)))&=&2p,\\
\label{2nd ptd}
\mbox{rank}(\Fun_{\Vec_G}(\M(P,\,1),\, \M(K,\, 1)))&=& 2p,\\
\label{3rd ptd}
\mbox{rank}(\Fun_{\Vec_G}(\M(H_1,\, 1),\,  \M(K,\, 1))) &=& 2,\\
\label{4th ptd}
\mbox{rank}(\Fun_{\Vec_G}(\M(H_2,\, 1),\,  \M(K,\, 1))) &=& 2,
\end{eqnarray*}
and the statement follows.
\end{proof}

Fix a primitive $p$-th root of unity $\xi$ in $k$. 
Any hyperbolic form $\chi$ on 
$A:=\mathbb{Z}/p\mathbb{Z} \times \mathbb{Z}/p\mathbb{Z} 
=\{(x,y)\mid x,y\in  \mathbb{Z}/p\mathbb{Z} \}$  is isomorphic to
\begin{equation}
\label{form chi}
\chi((x_1,x_2),\, (y_1, y_2)) = \xi^{x_1y_2+ y_1x_2},
\end{equation}
see, e.g., \cite{S}. There are exactly two Lagrangian subgroups of $A$:
\begin{equation}
\label{L1 and L2}
L_1=\{ (x,0) \mid  x\in  \mathbb{Z}/p\mathbb{Z} \}\quad \mbox{and} \quad
L_2 = \{ (0,y) \mid  y\in  \mathbb{Z}/p\mathbb{Z} \}.
\end{equation}

Observe that there is a $\mathbb{Z}/2\mathbb{Z}$-grading 
\[
\C_p = (\C_p)_0 \oplus (\C_p)_1,
\]
where the invertible objects of $\C_p$ span 
the trivial component $(\C_p)_0 \cong  \Vec_A$ 
and the unique non-invertible simple object spans $(\C_p)_1 \cong \Vec$.

The next two Lemmas describe the pointed $\C_p$-module categories 
from Proposition~\ref{Cp module categories} in terms of 
the explicit presentation of Tambara-Yamagami categories
given in Section~\ref{sect: definition of TY}. 


Recall that for a  subgroup $H \subset A$ and a $2$-cocycle
$\mu\in Z^2(H,\,k^\times)$ we defined the algebra 
$R(H,\,\mu) =\oplus_{a\in H}\, a$ in equation  \eqref{algebra R}
in Section~\ref{sect: group-theor}.

\begin{lemma}
\label{rank 2p}
Consider algebras  $R_i := R(L_i,\, 1),\, i=1,2$ in $\Vec_A \subset \C_p$.
The categories of $R_i$-modules, $i=1,2,$ in $\C_p$ are 
non-equivalent pointed $\C_p$-module categories of  rank $2p$.
\end{lemma}
\begin{proof}
Fix $i\in \{1,\,2\}$. There are  $p$ non-isomorphic simple
$R_i$-modules and $p^2$ non-isomorphic simple 
$R_i$-bimodules  in $(\C_p)_0$. The object $m$ has 
$p$ structures of an $R_i$-module and, hence,
$p^2$ non-isomorphic structures of an $R_i$-bimodule,
thanks to the associativity constraint property 
$\alpha_{a,m,b} = \id_m$ for all $a,b\in L_i$. 
Thus, the category of $R_{i}$-modules in $\C_p$ 
is a pointed $\C_p$-module category of rank $2p$. (see \cite{O2})

The two $\C_p$-module categories in question are not equivalent, since 
by Lemma~\ref{homogen} they are 
already not equivalent as $\Vec_A$-module categories.
\end{proof}

\begin{lemma}
\label{rank 2}
Let $\M$ be an indecomposable  $\C_p$-module category of rank $2$.
There is a  cohomologically non-trivial 
$2$-cocycle  $\mu \in Z^2(A,\, k^\times)$
such that $\M$ is equivalent to 
the category of $R(A,\, \mu)$-modules in $\C_p$. 
\end{lemma}
\begin{proof}
Let $x,\, y$ be simple objects of $\M$. It is easy to see 
that the fusion rules of $\M$ are
\[
x\ot a =x,\, y\ot a =y,\, x \ot m = py,\, y \ot m = px, \quad a\in A.
\]
Any such category $\M$ is equivalent to the category of ${B}$-modules
in $\C_p$, where ${B}= \underline{\Hom}(x,\,x)$ is the internal $\Hom$,
see Section~\ref{Sect: Prelim} and \cite{O1}. 
Thus, ${B} = \oplus_{a\in A}\, a$ as an object of $\C_p$. Hence,
${B} = R(A,\, \mu)$ for some $\mu \in Z^2(A,\, k^\times)$, 
where the algebra $R(A,\, \mu)$ is defined in \eqref{algebra R}.

Note that $\mu$ must be cohomologically non-trivial, since the object $m$
has $p^2$ structures of an $R(A,\, 1)$-module, 
and so the category of $R(A,\, 1)$-modules in $\C_p$ has rank $> 2$.
\end{proof}

\subsection{An action of $\mathbf{\mathbb{Z}/2\mathbb{Z}}$ on $\mathbf{\C_p}$
without invariant pointed $\mathbf{\C_p}$-module categories}

As before, let $A = \mathbb{Z}/p\mathbb{Z} \times \mathbb{Z}/p\mathbb{Z}$
and let $\chi: A \times A \to k^\times$ be a hyperbolic bilinear form
defined in \eqref{form chi}.  Let $t$ be the group automorphism
of $A$ defined by  
\begin{equation}
\label{t}
t(x,y) = (y,x),\qquad x,y\in \mathbb{Z}/p\mathbb{Z}.
\end{equation}
Then $t \in \Aut(A,\,\chi)$, i.e., $\chi\circ(t \times t) =\chi$.
By Proposition~\ref{orthogonal group acts} this 
$t$ gives rise to a tensor autoequivalence of $\C_p$ of order $2$
(which we will also denote $t$) and, hence,
to a  $\mathbb{Z}/2\mathbb{Z}$-action on $\C_p$.

\begin{proposition}
\label{no invariant}
The above action of $\mathbb{Z}/2\mathbb{Z}$ has no invariant
pointed $\C_p$-module categories.
\end{proposition}
\begin{proof}
By Proposition~\ref{Cp module categories}, there are exactly $4$ non-equivalent 
pointed $\C_p$-module categories. These categories are described in 
Lemmas~\ref{rank 2p} and \ref{rank 2} as categories of modules 
over certain algebras in $\C_p$.  

Note that $t$ permutes Lagrangian subgroups of $A$, 
i.e., maps $L_1$ to $L_2$ and vice versa. Hence, $t$ maps 
$R(L_1,\, 1)$ to $R(L_2,\, 1)$ 
and vice versa. It follows from Lemma~\ref{rank 2p} that $t$
permutes two pointed $\C_p$-module categories of rank $2p$.

Note that $t$ acts on $H^2(A, k^\times)\cong {\mathbb{Z}/p\mathbb{Z}}$ 
by taking  inverses, i.e., it maps
the cohomology class represented by a $2$-cocycle $\mu$ to that of $\mu^{-1}$.
In particular, the algebra $t(R(A,\, \mu))$
is isomorphic to $R(A,\, \mu^{-1})$.
We claim that $t$ permutes the two pointed
$\C_p$-module categories of rank $2$.
Indeed, let $\M$ be such a category.  By 
Lemma~\ref{rank 2}, $\M$ is equivalent to the category of
$R(A,\, \mu)$-modules in $\C_p$ for some cohomologically non-trivial 
$\mu\in Z^2(A, k^\times)$.  
By Lemma~\ref{dual preserved} $\M^t$ is pointed.
It is equivalent 
to the category of $R(A,\, \mu^{-1})$-modules. 
Considering $\M$ and $\M^t$ as $\Vec_A$-module categories we have,
using Lemma~\ref{homogen}:
\begin{equation}
\label{restriction M}
\M \cong  \M(A,\mu) \oplus \M(A,\mu)
\not \cong  \M(A,\mu^{-1}) \oplus \M(A,\mu^{-1}) \cong \M^t,
\end{equation}
where $\M(A, \mu)$ denotes the $\Vec_A$-module category 
of $R(A,\, \mu)$-modules in $\Vec_A$ 
described in Section~\ref{Sect: Prelim}. 
This means that $\M$ and $\M^t$
are non-equivalent as $\Vec_A$-module categories.
Hence, they are not equivalent as $\C_p$-module categories.
\end{proof}

\begin{remark}
For any cohomologically non-trivial $\mu \in Z^2(A,\,k^\times)$ let
$\N_\mu$ denote the category of $R(A,\,\mu)$-modules in $\C_p$. 
As a $\Vec_A$-module category $\N_\mu$ decomposes as
\[
N_\mu = \M(A,\,\mu) \oplus \M(A,\mu'),
\]
where the cohomology class of $\mu'\in Z^2(A,\,k^\times)$
depends on that of $\mu$ as follows. Note that 
$\mbox{Alt}(\mu)(x,\,y):= {\mu(y,\,x)}\mu(x,\,y)^{-1},\, x,y\in A$ is a non-degenerate
alternating bilinear form. There is a unique group 
automorphism $\iota_\mu\in \Aut(A)$ defined by the property 
\[
\mbox{Alt}(\mu)(x,\,\iota_\mu(a)) = \chi(x,\, a),\qquad \mbox{ for all } x\in A. 
\]
Then $\mu'(x,y) = \mu^{-1}(\iota_\mu(x),\, \iota_\mu(y)),\, x,y\in A$. 
It can be checked directly that there are exactly two cohomology classes $\mu$
with the property that $\mu$ and $\mu'$ are cohomologous and that these
two classes are inverses of each other. This fact is not needed in the proof
of Proposition~\ref{no invariant} but rather explains the nature of 
the cohomology classes of cocycles $\mu,\,\mu^{-1}\in Z^2(A,\,k^\times)$
corresponding to pointed $\C_p$-module categories of rank $2$. 
\end{remark}

\begin{corollary}
The category $(\C_p)^{\mathbb{Z}/2\mathbb{Z}}$ corresponding to
the action \eqref{t} is a non group-theoretical fusion category 
and  is equivalent to the representation category a semisimple
Hopf algebra of dimension $4p^2$.
\end{corollary}
\begin{proof}
The category $(\C_p)^{\mathbb{Z}/2\mathbb{Z}}$  is non
group-theoretical by Corollary~\ref{criterion for CG to be gr th} 
and Proposition~\ref{no invariant}.  Since $\C_p$ has a fiber
functor, then so does its equivariantization, see Remark~\ref{fiber}.
By Tannakian reconstruction theorem \cite{U}, 
$(\C_p)^{\mathbb{Z}/2\mathbb{Z}}$ 
is equivalent to the representation category of a semisimple Hopf algebra.
\end{proof}

\section{Non group-theoretical 
semisimple Hopf algebras of dimension ${4p^2}$ as extensions}
\label{Hopf}

Let $G$ be a finite group. Below $k^G$ will denote the commutative
Hopf algebra of functions on $G$ and $kG$ will denote the cocommutative
group Hopf algebra of $G$.  For a Hopf algebra $H$ its representation 
category will be denoted $\Rep(H)$.

We will freely use Hopf algebra
notation and terminology, see \cite{M} as a reference.

Let $H_p$ be a semisimple Hopf algebra such that $\Rep(H_p) \cong 
(\C_p)^{\mathbb{Z}/2\mathbb{Z}}$ as a fusion  category. We have
$\dim_k(H_p)=4p^2$.
In Proposition~\ref{describe Hp} below we find the algebra structure of  
$H_p$ and its dual $H_p^*$ and describe them in terms of Hopf 
algebra extensions.  

\begin{remark}
\label{natale}
It was shown by S.~Natale \cite{Nt} that a semisimple Hopf
algebra $H$ for which there is a short exact sequence of Hopf algebras
\[
k \to k^{G_1} \to H \to k{G_2} \to k,
\]
where $G_1,\, G_2$ are finite groups, then $\Rep(H)$ is group-theoretical.
Thus, $H_p$ cannot be obtained as an extension of a cocommutative Hopf algebra
by a commutative one.
\end{remark}

It was shown by A.~Masuoka \cite{Ma} that there is a unique, up to an 
isomorphism, semisimple Hopf algebra $A_p$ of dimension $2p^2$
with exactly $p^2$ group-like elements. This Hopf algebra is dual to 
the Kac-Paljutkin Hopf algebra \cite{KP} and 
there is a short exact sequence of Hopf algebras
\[
k \to k^{\mathbb{Z}/p\mathbb{Z} \times \mathbb{Z}/p\mathbb{Z}} \to A_p 
\to k (\mathbb{Z}/2\mathbb{Z}) \to k.
\]
As algebras,
\[
A_p \cong k^{(2p)} \oplus M_2(k)^{(\tfrac{p(p-1)}{2})}
\quad \mbox{ and } \quad
A_p^* \cong k^{(p^2)} \oplus M_p(k),
\]
where $M_n(k)$ denotes the algebra of $n$-by-$n$
matrices over $k$.

\begin{proposition}
\label{describe Hp}
Both $H_p$ and $H_p^*$ are extensions of $A_p$ by $k^{\mathbb{Z}/2\mathbb{Z}}$, i.e.,
each of them  fits into a short exact sequence of Hopf algebras
\begin{equation}
\label{short}
k \to k^{\mathbb{Z}/2\mathbb{Z}} \to H \to A_p \to k,
\end{equation}
where $H = H_p$ or $H=H_p^*$.

As algebras, $H_p \cong H_p^* \cong
k^{(2p)} \oplus M_2(k)^{(\tfrac{p(p-1)}{2})} \oplus M_p(k)^{(2)}$.
\end{proposition}
\begin{proof}
The isomorphism type of the algebra $H_p$ is found directly by
computing Frobenius-Perron dimensions of simple objects
of $\Rep(H_p) \cong (\C_p)^{\mathbb{Z}/2\mathbb{Z}}$.

Since the fusion category $(\C_p)^{\mathbb{Z}/2\mathbb{Z}}$
inherits a $\mathbb{Z}/2\mathbb{Z}$-grading from $\C_p$, the Hopf algebra $H_p$
contains a central group-like element of order $2$ and, hence, a central
Hopf subalgebra $k^{\mathbb{Z}/2\mathbb{Z}}$. The representation category
of the quotient Hopf algebra $\overline{H}_p = H/H (k^{\mathbb{Z}/2\mathbb{Z}})^+$
is equivalent to  a  $\mathbb{Z}/2\mathbb{Z}$-equivariantization of 
the fusion subcategory 
$\Vec_{\mathbb{Z}/p\mathbb{Z} \times \mathbb{Z}/p\mathbb{Z}}
\cong \Rep(\mathbb{Z}/p\mathbb{Z} \times \mathbb{Z}/p\mathbb{Z})$ of $\C_p$.
By \eqref{t} the latter comes from an order $2$ group automorphism
$t$ of $\mathbb{Z}/p\mathbb{Z} \times \mathbb{Z}/p\mathbb{Z}$.
From Example~\ref{crossed product} we see that 
$\Rep(\overline{H}_p ) \cong \Rep(kG)$, where 
$G = (\mathbb{Z}/p\mathbb{Z} \times \mathbb{Z}/p\mathbb{Z})\rtimes
\mathbb{Z}/2\mathbb{Z}$.
It follows from the result of P.~Schauenburg \cite{Sch} that  there is cocycle 
twist $J$ on the group Hopf algebra $kG$ such that $H_p$ is isomorphic to
the deformation $(kG)^J$ of $kG$ by means of $J$.

Thus, $H_p$ fits into an extension
\begin{equation}
\label{Hp}
k \to k^{\mathbb{Z}/2\mathbb{Z}} \to H_p \to (kG)^J \to k.
\end{equation}
Note that the Hopf algebra  $(kG)^J$ is necessarily non cocommutative,
since otherwise $H_p$ would be group-theoretical by the result of S.~Natale 
\cite{Nt}, see Remark~\ref{natale}. Hence, the cocycle twist $J$ must be  non-trivial.
It follows from the work of P.~Etingof and S.~Gelaki \cite{EG} 
that $J$ comes from a non-degenerate $2$-cocycle on the Sylow $p$-subgroup of $G$.
Therefore, the deformed Hopf algebra $(kG)^J$ has exactly $p^2$ group-like elements.
Hence, $(kG)^J \cong A_p$ by \cite{Ma}.

Thus, $A_p^*$
is an index $2$ (and, hence, normal) Hopf subalgebra of $H_p^*$.
Therefore,  $H_p^*$ has exactly two $p$-dimensional irreducible representations
and the central group-like element of order $2$ in $A_p^*$ must also be central
in $H_p^*$. The corresponding 
quotient Hopf algebra $H_p^*/H_p^* (k^{\mathbb{Z}/2\mathbb{Z}})^+$
is non cocommutative and contains a Hopf subalgebra isomorphic to
$A_p^*/A_p^*   (k^{\mathbb{Z}/2\mathbb{Z}})^+ \cong
k(\mathbb{Z}/p\mathbb{Z} \times \mathbb{Z}/p\mathbb{Z})$. 
Again, $H_p^*/H_p^* (k^{\mathbb{Z}/2\mathbb{Z}})^+\cong A_p$ 
by \cite{Ma} and so  $H_p^*$ fits into the same extension \eqref{Hp} as $H_p$
and has the same algebra structure. 
\end{proof}


\begin{corollary}
The Hopf algebras $H_p$ are
upper and lower semisolvable in the sense of 
S.~Montgomery and S.~Witherspoon \cite{MW}.
\end{corollary}


\bibliographystyle{ams-alpha}

\end{document}